\documentclass[12pt]{amsart}
\usepackage{amsmath,amssymb,amsthm}
\usepackage[dvips]{graphicx}
\usepackage{color}
\usepackage{hyperref}
\usepackage{wasysym}
\newtheorem{lem}{Lemma}[section]
\newtheorem{prop}{Proposition}[section]
\newtheorem{cor}{Corollary}[section]

\newtheorem{thm}{Theorem}[section]

\theoremstyle{definition}

\theoremstyle{remark}

\theoremstyle{remark}
\newtheorem{remark}{Remark}[section]
\numberwithin{equation}{section}

\newcommand{\R}{{\mathbb R}}

\definecolor{blu}{rgb}{0,0,1}

\begin{document}
\title[Sharp  lower bounds for Coulomb energy]{Sharp lower bounds for Coulomb energy}

\author{Jacopo Bellazzini}
\address{J. Bellazzini, 
\newline  Universit\`a di Sassari, Via Piandanna 4, 07100 Sassari, Italy}%
\email{jbellazzini@uniss.it}%
\author{Marco Ghimenti}
\address{M. Ghimenti, \newline Dipartimento di Matematica Universit\`a di Pisa
Largo B. Pontecorvo 5, 56100 Pisa, Italy}%
\email{ghimenti@mail.dm.unipi.it}
\author{Tohru Ozawa}
\address{T. Ozawa
\newline Department of Applied Physics, Waseda University, Tokyo 169-8555, Japan}%
\email{txozawa@waseda.jp }%
\maketitle

\begin{abstract}
We prove $L^p$ lower bounds for Coulomb energy for radially symmetric functions in $\dot H^s(\R^3)$ with $\frac 12 <s<\frac{3}{2}$. In case 
$\frac 12 <s \leq 1$ we show that the lower bounds are sharp.
\end{abstract}
In this paper  we prove lower bounds for the Coulomb energy 
$$\iint_{\R^3\times \R^3}
\frac{|\varphi(x)|^2 |\varphi(y)|^2}{|x-y|} \, dxdy$$
if radial symmetry of $\varphi$ is assumed.\\
In the general case, without restricting to radial functions, the upper bound for the Coulomb energy is given by the
well known Hardy-Littlewood-Sobolev inequality while lower bounds have been proved only very recently. In particular if
one can control suitable homogeneous Sobolev space $\dot H^{s}(\R^3)$ 
the $L^p$ lower bound for  the Coulomb energy is given by the following inequalities
\begin{equation}
\label{eq:dbound}
\|\varphi\|_{L^{2p}(\R^3)}\leq C(p,s) \|\varphi\|_{\dot H^{s}(\R^3)}^{\frac{\theta}{2-\theta}}\left(\iint_{\R^3\times \R^3}
\frac{|\varphi(x)|^2 |\varphi(y)|^2}{|x-y|} \, dxdy\right)^{\frac{1-\theta}{4-2\theta}}
\end{equation}
with $\theta=\frac{6-5p}{3-2ps-2p}$. Here the parameters $s>0$ and $1<p\leq\infty$ satisfy
\begin{align*}
& p\in \left[\frac{3}{3-2s}, \frac{1+2s}{1 +s}\right] & & \text{if}\ 0<s<1/4 \,, \\
& p = \frac{3}{3-2s}= \frac{1+2s}{1 +s} & & \text{if}\ s=1/4 \,, \\
& p\in\left[ \frac{1+2s}{1 +s},  \frac 3{3-2s}\right] & & \text{if}\ 1/4< s<3/2 \,, \\
& p\in \left[\frac{1+2s}{1 +s}, \infty\right) & & \text{if}\ s= 3/2 \,, \\
& p\in \left[\frac{1+2s}{1+s}, \infty\right] & & \text{if}\ s> 3/2 \,.
\end{align*}
These bounds have been proved in \cite{BFV} while the case $s=\frac 12$ has been first considered in \cite{BOV}.
These bounds follows from a suitable Gagliardo-Nirenberg inequality, see Theorem 2.44 of \cite{BCD}, together with the 
following well known identity
$$\iint_{\R^3\times \R^3} \frac{|\varphi(x)|^2 |\varphi(y)|^2}{|x-y|} dxdy =c ||\varphi ^2||_{\dot H^{-1}(\R^3)}^2.$$ 
We shall underline that in many physical applications involving Sobolev norms and Coulomb energy the radially symmetric assumption of $\varphi$ is natural due to the rotational invariance of energy functionals (see e.g \cite{LS} in the context of stability of matter). Our purpose is to see if it is possible to control lower $L^p$ norms if one assumes \emph{radial symmetry} of $\varphi$.\\
In the sequel we use two theorems that are crucial for our improvement  in case of radial symmetry. 
The first is the following pointwise decay for radial functions in $\dot H^{s}(\R^d)\cap L^q_a(\R^d)$, see \cite{D},
where $ L^q_a(\R^d)$ is the weighted Lebesgue space with the norm
$$||u||_{L^q_a(\R^d)}=\left(\int_{\R^d} |x|^a|u|^qdx\right)^{\frac{1}{q}}$$
\begin{thm}[De N\'{a}poli  \cite{D}]\label{thm:denapoli} Let $\varphi$ be a radial function in $\dot H^{s}(\R^d)\cap L^q_a(\R^d)$ with $s>\frac 12$ and $-(d-1)<a<d(q-1)$, then
$$ |\varphi(x)|\leq C(d,s,q,a)|x|^{-\sigma} ||(-\Delta)^\frac{s}{2} \varphi||_{L^2(\R^d)}^{\theta}||\varphi|||_{L^q_a(\R^d)}^{1-\theta}$$
where $\theta=\frac{2}{2sq+2-q}$, $\sigma=\frac{2as+2ds-a-2s}{2sq+2-q} $.
\end{thm}
\begin{remark}
The strategy of the proof  of Theorem \ref{thm:denapoli}  is based on Fourier representation for radial functions in $\R^d$ (identifying the function with its profile)
$$
\varphi (x)=(2\pi)^{\frac{d}{2}}|x|^{-\frac{d-2}{2}}\int_0^{\infty} J_{\frac{d-2}{2}}(|x|\rho)\hat \varphi (\rho)\rho^{\frac{d}{2}}d\rho
$$
where $J_{\frac{d-2}{2}}$ is the Bessel function of order $\frac{d-2}{2}.$ The argument is similar to the one developed in \cite{CO} for  the pointwise decay of radial function in $\dot H^s(\R^d)$, i.e to split $\varphi$ into low and high frequency parts. The pointwise decay of high frequency part of $\varphi$ will be controlled by the boundess of Sobolev norm while the decay of low frequency  part
by  the boundness of the weighted Lebesgue norm.
\end{remark}
The second theorem is the following lower bound for the Coulomb energy by Ruiz, see \cite{R}.
\begin{thm}[Ruiz \cite{R}]\label{thm:ruiz}
Given $\alpha>\frac{1}{2}$, there exists $c=c(\alpha)>0$ such that for any measurable $\varphi:\R^d \rightarrow \R$ we have
$$\iint_{\R^d\times \R^d}
\frac{|\varphi(x)|^2 |\varphi(y)|^2}{|x-y|^{d-2}} dxdy \geq c \left( \int_{\R^d} \frac{|\varphi(x)|^2}{|x|^{\frac{d-2}{2}}(1+|\log |x ||)^{\alpha}}dx \right)^2.$$
\end{thm}
Let us define
$$\mathcal{E}^s=\{ \varphi \in  \dot H^s_{rad}(\R^3)  \  \ s.t  \ \iint_{\R^3\times \R^3}
\frac{|\varphi(x)|^2 |\varphi(y)|^2}{|x-y|} dxdy <\infty \}$$
with 
$$||\varphi||_{\mathcal{E}^s}=\left(||\varphi||_{\dot H^s(\R^3)}^2+\left(\iint_{\R^3\times \R^3}
\frac{|\varphi(x)|^2 |\varphi(y)|^2}{|x-y|} dxdy\right)^{\frac 12} \right)^{\frac 12}.$$
Following the argument of Ruiz \cite{R} it is easy to show that $||\cdot||_{\mathcal{E}^s}$ is a norm and $C^{\infty}_0(\R^3)$ is dense in $\mathcal{E}^s$.
In \cite{R} Ruiz  proved that for $\mathcal{E}^1$ the following continuous embedding
$$\mathcal{E}^1  \hookrightarrow  L^p \ \ \ p\in(\frac{18}{7},6].$$
The result by Ruiz follows from two steps: first,  Theorem \ref{thm:ruiz} proves that $\mathcal{E}^1\subset \dot H^1_{rad}(\R^3) \cap L^2(\R^3, V(x)dx)$ where
$V(x)=\frac{1}{(1+|x|)^{\gamma}}$ with $\gamma>\frac{1}{2}$, second, a weighted Sobolev embedding for radial function proved
by Su, Wang and Wilem \cite{SWW} gives the inclusion
$$\dot H^1_{rad}(\R^3) \cap  L^2(\R^3, V(x)dx) \subset L^q(\R^3) \ \ \ q\in [\frac{2(4+\gamma)}{4-\gamma},6]$$

The aim of our paper is to find continuous embeddings and hence better lower bounds for the Coulomb energy assuming radial symmetry when $\frac 12<s<\frac{3}{2}$. As a particular case we recover
$p=\frac{18}{7}$ as end-point exponent when $s=1$.
\begin{thm}\label{thm:nonsharp}
$\mathcal{E}^s \hookrightarrow L^p(\R^3)$ continuously for 
\begin{align*}
& p\in\left( \frac{16s+2}{6s+1},  \frac 6{3-2s}\right] & & \text{if}\ 1/2< s<3/2 . \\
\end{align*}
\end{thm}

The above result is sharp when $\frac{1}{2}<s\leq 1$ as showed by the following 

\begin{thm}\label{thm:sharp}
Let  $\frac{1}{2}<s\leq 1$, then the space $\mathcal{E}^{s}$ is not embedded in
$L^{p}$ for $p<\frac{16s+2}{6s+1}$.\end{thm}

From the continuous embedding for $\mathcal{E}^s$ it is elementary to derive the scaling invariant lower bounds for the Coulomb energy given by
\eqref{eq:dbound} for  $p\in\left( \frac{16s+2}{6s+1},  \frac 6{3-2s}\right]$  and $\frac{1}{2}<s< \frac 32$.
Moreover we prove that the best constants of the  lower bounds  are  achieved among radially symmetric functions.
\begin{cor}\label{corr}
Let $\varphi$ be radially symmetric, then the following scaling invariant inequality holds
\begin{equation}
\label{eq:dbound2}
\|\varphi\|_{L^{2p}(\R^3)}\leq C(p,s) \|\varphi\|_{\dot H^{s}(\R^3)}^{\frac{\theta}{2-\theta}}\left(\iint_{\R^3\times \R^3}
\frac{|\varphi(x)|^2 |\varphi(y)|^2}{|x-y|} \, dxdy\right)^{\frac{1-\theta}{4-2\theta}}
\end{equation}
with $\theta=\frac{6-5p}{3-2ps-2p}$. Here the parameters $s$ and $p$ satisfy
$$ p\in (\frac{8s+1}{6s+1},  \frac 3{3-2s}]  \ \ \text{if}\ 1/2< s<3/2. $$
Assume, moreover, that $p\neq \frac 3{3-2s}$, then the best constant in \eqref{eq:dbound2} is achieved in the set of radially symmetric functions.

\end{cor}

In Figure 1 the behavior of $p$ as a function of $s$ is plotted. \\
\\
{\bf Acknowledgement:} the authors thanks  Nicola Visciglia for fruitful conversations. The first author thanks also Rupert Frank and Elliot Lieb for interesting discussions
around the problem.

\begin{figure}\label{figu}
{\includegraphics[width=8cm,height=4.5cm]{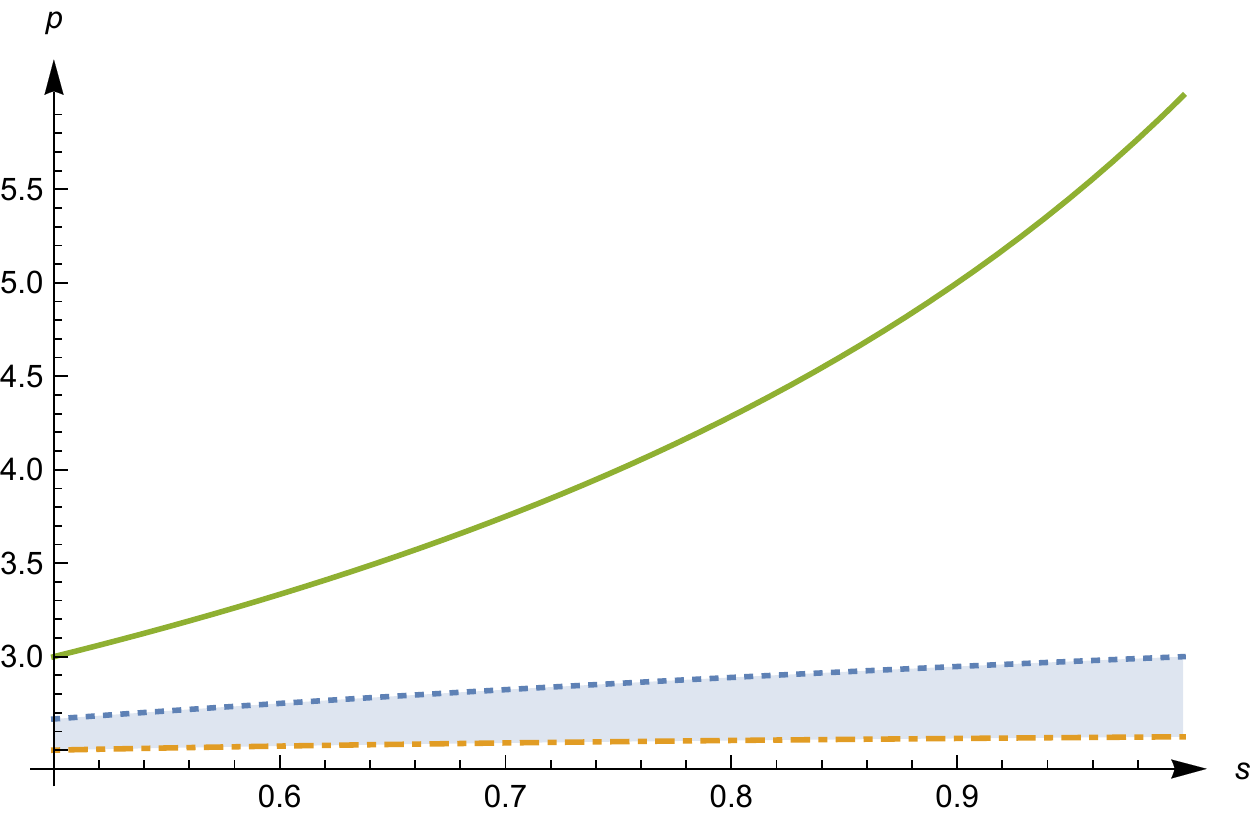}}
	 \caption{$L^p$ lower bound for Coulomb energy: without radial symmetry the lower bound is $\frac{2+4s}{1 +s}$(dotted) and for the radially symmetric case is $\frac{16s+2}{6s+1}$ (dash-dotted). The bold line plots the Sobolev embedding exponent for $\frac 12<s\leq 1.$}
\end{figure}
\section{Proof of Theorem \ref{thm:nonsharp}}
\begin{prop}\label{facile}
Let $\gamma>\frac{1}{2}$ and $\frac{\gamma}{2}<s<\frac{3}{2}$ then exists $c(\gamma,s )>0$ such that for any $\varphi \in \mathcal{E}^s$
$$\left(\int_{\R^3} |x|^{-\gamma}|\varphi |^2dx\right)\leq c(\gamma,s) ||\varphi||_{\mathcal{E}^s}^2.$$
\end{prop}
\begin{proof}
By elementary computation we notice that if $2s>\gamma$
$$\left(\int_{\R^3} |x|^{-\gamma}|\varphi |^2dx\right)\leq R^{2s-\gamma}\left(\int_{B(0,R)}\frac{|\varphi |^2}{|x|^{2s}}dx\right)+ \frac{(1+R)^{\gamma}}{R^{\gamma}}\left(\int_{B(0,R)^c}\frac{|\varphi |^2}{(1+|x|)^{\gamma}}dx\right).$$
On the other hand, if $0<s<\frac{3}{2}$, by Pitt inequality \cite{B},
$$c_s||\varphi||_{\dot H^s(\R^3)}^2=c_s\left(\int_{\R^3}|\hat \varphi |^2|\xi|^{2s}d\xi\right)\geq \left(\int_{\R^3}\frac{| \varphi |^2}{|x|^{2s}}d\xi\right)$$
where $c_s=\pi^{2s}\left[\frac{\Gamma(\frac{3-2s}{4})}{\Gamma(\frac{3+2s}{4})}\right]^2$
and by Ruiz's Theorem \ref{thm:ruiz}
$$\left(\int_{\R^3}\frac{|\varphi |^2}{(1+|x|)^{\gamma}}dx\right)\leq c(\gamma) \left(\iint_{\R^3\times \R^3}
\frac{|\varphi(x)|^2 |\varphi(y)|^2}{|x-y|} dxdy\right)^{\frac 12}.$$
\end{proof}
\begin{proof}[Proof of Theorem \ref{thm:nonsharp}]$$$$
By  Theorem \ref{thm:ruiz} we have that for any $\alpha>\frac{1}{2}$, $\mathcal{E}^s\subset L^2(\R^3, V(x)dx)$ where $V(x)=(\frac{1}{1+|x|})^{\gamma}$, $\gamma> \frac 12$.
Let us call $p^{*}=\frac{2+4s}{1 +s}$, the end-point  exponent for \eqref{eq:dbound}.  H\"older inequality assures that  for $p<p^{*}$
\begin{eqnarray}
\int_{B(0,1)} |\varphi|^p dx<\mu(B(0,1))^{\frac{p^{*}-p}{p}}\left(\int_{B(0,1)} |\varphi|^{p^{*}}dx\right)^{\frac{p}{p^{*}}}\leq  \nonumber \\
\leq  C \|\varphi\|_{\dot H^{s}(\R^3)}^{(\frac{\theta}{2-\theta})p}\left(\iint_{\R^3\times \R^3} \nonumber
\frac{|\varphi(x)|^2 |\varphi(y)|^2}{|x-y|} \, dxdy\right)^{(\frac{1-\theta}{4-2\theta})p}
\end{eqnarray}
with $\theta=\frac{6-\frac 52 p^{*}}{3-p^{*}s-p^{*}}$. On the other  hand by Proposition \ref{facile} and the radial decay
given by Theorem \ref{thm:denapoli} choosing $a=-\gamma$, $q=2$ and $d=3$,
\begin{eqnarray}
\int_{B(0,1)^c} |\varphi|^p dx=\int_{B(0,1)^c}  |\varphi(x)|^2|\varphi(x)|^{p-2}dx \leq \\ \nonumber
\leq C ||\varphi||_{\dot H^s(\R^3)}^{\theta(p-2)}||\varphi|||_{L^2_{-\gamma}(\R^3)}^{(1-\theta)(p-2)}  \int_{B(0,1)^c}   |x|^{-\sigma(p-2)}|\varphi(x)|^2dx \nonumber
\end{eqnarray}
where  $\theta=\frac{1}{2s}$, $\sigma=\frac{-2\gamma s+4s+\gamma}{4s}. $
Now 
$$\lim_{\gamma\rightarrow \frac 12} \frac{-2\gamma s+4s+\gamma}{4s}(p-2)=(\frac{3s+\frac{1}{2}}{4s})(p-2)$$
and this implies again by Proposition \ref{facile} that
$$\int_{B(0,1)^c} |\varphi|^p dx<+\infty$$
provided that
$(\frac{3s+\frac{1}{2}}{4s})(p-2)>\frac 12$, i.e if $p>\frac{16s+2}{6s+1}.$
\end{proof}

\section{Proof of Theorem \ref{thm:sharp}}
The proof of Theorem \ref{thm:sharp} is obtained constructing a counterexample, i.e 
a function $u$ such that 
\begin{align}
\nonumber &\|u\|_{\dot H^{s}(\R^3)}^{2}  \simeq 1\\
\label{eq:claim}&\iint_{\R^3\times \R^3}\frac{|u(x)|^2|u(y)|^{2}}{|x-y|} dxdy \simeq  1\\
\nonumber &||u||_{L^p(\R^3)}^{p}  \rightarrow +\infty
\end{align}

\begin{proof}[Proof of Theorem \ref{thm:sharp}.]$$$$
The case $s=1$ has been proved by Ruiz \cite{R}.
Set  $u:\mathbb{R}^{3}\rightarrow\mathbb{R}^{+}$
\begin{equation}\label{eq:u}
u(x)=\left\{ \begin{array}{cc}
\displaystyle \varepsilon\frac{ S-\big| |x|-R\big|}{S} & \text{ for }\big||x|-R\big|<S\\
&\\
0 & \text{ elsewhere}
\end{array}\right.
\end{equation}
where $R>S>>1>>\varepsilon>0$ will be precised in the sequel. 

We recall, by Ruiz \cite[Section 4]{R}, that
\[
\iint_{\R^3\times \R^3}\frac{|u(x)|^2|u(y)|^{2}}{|x-y|}dxdy \leq C \varepsilon^{4}S^{2}R^{3}
\]
and 
\[
||u||_{L^{p}(\R^3)}^{p}\ge C \varepsilon^{p}SR^{2}.
\]
Moreover we have
\begin{equation}\label{eq:Hs}
\|u\|_{\dot{H}^{s}(\R^3)}^{2}\leq C \frac{\varepsilon^{2}R^{2}}{S^{2s-1}}.
\end{equation}
The proof of (\ref{eq:Hs}) is not difficult and it will postponed to Lemma \ref{lem:stimaHs}. 

In order to have $\|u\|_{\dot{H}^{s}(\R^3)}^{2}\simeq 1$ we choose $S=\varepsilon^{\frac{2}{2s-1}}R^{\frac{2}{2s-1}}$.
At this point we have 
\[
\iint_{\R^3\times \R^3}\frac{|u(x)|^2|u(y)|^{2}}{|x-y|}dxdy\le C \varepsilon^{4}\varepsilon^{\frac{4}{2s-1}}R^{\frac{4}{2s-1}}R^{3}=C\varepsilon^{\frac{8s}{2s-1}}R^{\frac{6s+1}{2s-1}}
\]
and we choose $R=\varepsilon^{-\frac{8s}{6s+1}}$ to have the Coulomb norm bounded, so 
\[
S=\varepsilon^{\frac{2}{2s-1}}R^{\frac{2}{2s-1}}=\varepsilon^{\frac{2}{2s-1}}\varepsilon^{-\frac{8s}{6s+1}\frac{2}{2s-1}}=\varepsilon^{\frac{2-4s}{(6s+1)(2s-1)}}=\varepsilon^{-\frac{2}{6s+1}}.
\]
We remark that, since $s>1/2,$ then $8s>2$ and $R>S$, as required in the definition of $u(x)$. 

Concluding, we have
\[
||u||_{L^p(\R^3)}^{p}\ge C \varepsilon^{p}SR^{2}\simeq \varepsilon^{p-\frac{16s+2}{6s+1}}
\]
that diverges for $p<\frac{16s+2}{6s+1}$ when $\varepsilon\rightarrow 0$.
The claim follows immediately.
\end{proof}

\begin{lem}\label{lem:stimaHs}
Let $u$ be defined in (\ref{eq:u}). Then
$$\|u\|_{\dot{H}^{s}(\R^3)}^{2}\leq C \frac{\varepsilon^{2}R^{2}}{S^{2s-1}}.$$
\end{lem}
\begin{proof}
We want to compute the $\dot{H}^{s}$ norm of $u$ for $s<1$, that
is
\[
\|u\|_{\dot{H}^{s}(\R^3)}^{2}=C(s)\int\limits _{\mathbb{R}^{3}}\int\limits _{\mathbb{R}^{3}}\frac{|u(x)-u(y)|^{2}}{|x-y|^{3+2s}}dxdy.
\]
where $C(s)=2^{2s-1}\pi^{-\frac{3}{2}}\frac{\Gamma(\frac{3+2s}{2})}{\Gamma(-s)}$. \\
We observe that $u(x)-u(y)\neq0$ in the following five subsets of
$\mathbb{R}^{3}\times\mathbb{R}^{3}$: 
\begin{align*}
A_{1} & =\left\{ R-S\le|y|\le R+S,\ |x|\le R-S\right\} \\
A_{2} & =\left\{ R-S\le|y|\le R+S,\ |x|\ge R+S\right\} \\
A_{3} & =\left\{ R-S\le|x|\le R+S,\ |y|\le R-S\right\} \\
A_{4} & =\left\{ R-S\le|x|\le R+S,\ |y|\ge R+S\right\} \\
A_{5} & =\left\{ R-S\le|x|\le R+S,\ R-S\le|y|\le R+S\right\} 
\end{align*}
and, by symmetry, we obtain
\begin{align*}
\frac{\|u\|_{\dot{H}^{s}(\R^3)}^{2}}{C(s)}= & \int\limits _{\mathbb{R}^{3}}\int\limits _{\mathbb{R}^{3}}\frac{|u(x)-u(y)|^{2}}{|x-y|^{3+2s}}dxdy=2\iint_{A_{1}}\frac{|u(x)-u(y)|^{2}}{|x-y|^{3+2s}}dxdy\\
 & +2\iint_{A_{2}}\frac{|u(x)-u(y)|^{2}}{|x-y|^{3+2s}}dxdy+\iint_{A_{5}}\frac{|u(x)-u(y)|^{2}}{|x-y|^{3+2s}}dxdy\\
\le & 2\int\limits _{R-S\le|y|\le R+S}\ \int\limits _{\mathbb{R}^{3}}\frac{|u(x)-u(y)|^{2}}{|x-y|^{3+2s}}dxdy.
\end{align*}
Therefore, 
\begin{align*}
\|u\|_{\dot{H}^{s}(\R^3)}^{2} \lesssim & \int\limits _{R-S\le|y|\le R+S}\ \int\limits _{|x-y|\le S}\frac{|u(x)-u(y)|^{2}}{|x-y|^{3+2s}}dxdy\\
 & +\int\limits _{R-S\le|y|\le R+S}\ \int\limits _{|x-y|\ge S}\frac{|u(x)-u(y)|^{2}}{|x-y|^{3+2s}}dxdy\\
\apprle & \int\limits _{R-S\le|y|\le R+S}\ \int\limits _{|x-y|\le S}\ \frac{\varepsilon^{2}}{S^{2}}\frac{|x-y|^{2}}{|x-y|^{3+2s}}dxdy\\
 & +\int\limits _{R-S\le|y|\le R+S}\ \int\limits _{|x-y|\ge S}\frac{\varepsilon^{2}}{|x-y|^{3+2s}}dxdy
\end{align*}
using that $|u(x)-u(y)|\le\sup_{\xi}|\nabla u(\xi)||x-y|\le\frac{\varepsilon}{S}|x-y|$
in the first term and that $|u(x)|\le\varepsilon$ in the second term.
At this point, with the change of variable $t=x-y$ we get 
\begin{align*}
\|u\|_{\dot{H}^{s}(\R^3)}^{2}\apprle & \varepsilon^{2}\int\limits _{R-S\le|y|\le R+S}\left[\int\limits _{|t|\le S}\ \frac{1}{S^{2}}\frac{1}{|t|^{1+2s}}dt+\int\limits _{|t|\ge S}\ \frac{1}{|t|^{3+2s}}dt\right]dy\\
\apprle & \varepsilon^{2}\int\limits _{R-S\le|y|\le R+S}\left[\int\limits _{0}^{S}\ \frac{1}{S^{2}}\frac{r^{2}}{r^{1+2s}}dr+\int\limits _{S}^{\infty}\ \frac{r^{2}}{r^{3+2s}}dr\right]dy\\
\apprle & \varepsilon^{2}R^{2}S\left[\int\limits _{0}^{S}\ \frac{1}{S^{2}}r^{1-2s}dr+\int\limits _{S}^{\infty}r^{-1-2s}dr\right]\simeq\varepsilon^{2}R^{2}S\left[\frac{S^{2-2s}}{S^{2}}+S^{-2s}\right]\\
\simeq & \frac{\varepsilon^{2}R^{2}}{S^{2s-1}}
\end{align*}

\end{proof}

\begin{proof}[Proof of Corollary \eqref{corr}]$$$$
From Theorem \ref{thm:nonsharp} it follows that
$$||\varphi_{\lambda}||_{L^p(\R^3)}^2\leq C \left(||\varphi_{\lambda}||_{\dot H^s(\R^3)}^2+\left(\iint_{\R^3\times \R^3}
\frac{|\varphi_{\lambda}(x)|^2 |\varphi_{\lambda}(y)|^2}{|x-y|} dxdy\right)^{\frac 12} \right)$$
if  $p\in\left( \frac{16s+2}{6s+1},  \ \  \frac 6{3-2s}\right] $ and  $\text{if}\ 1/2< s<3/2 $.  
Now let us consider the following scaling $$\varphi_{\lambda}=\lambda^{\frac{3}{p}}\varphi(\lambda x),$$
such that $||\varphi_{\lambda}||_{L^p(\R^3)}=||\varphi||_{L^p(\R^3)}$ for all $\lambda>0$. By elementary computation one gets
$$||\varphi||_{L^p(\R^3)}^2\leq C \left(\lambda^{\frac{6}{p}-(3-2s)}||\varphi||_{\dot H^s(\R^3)}^2+\lambda^{\frac{6}{p}-\frac 52}\left(\iint_{\R^3\times \R^3}
\frac{|\varphi(x)|^2 |\varphi(y)|^2}{|x-y|} dxdy\right)^{\frac 12} \right).$$
and  minimizing R.H.S we get the desired inequality. \\
The argument to show the existence of maximizers is identical to the one used to show Theorem 2.2 in \cite{BFV}. We just give a sketch of the proof for reader convenience. Let us fix $p$ in the set $( \frac{16s+2}{6s+1},  \frac 6{3-2s})$.
By homogeneity and scaling we can assume that an optimizing sequence $\varphi_n\in\mathcal E^s$ satisfies
$$
\|\varphi_n \|_{\dot H^s}=
\iint_{\R^3\times\R^3} \frac{|\varphi_n(x)|^2\ |\varphi_n(y)|^2}{|x-y|} \,dx\,dy = 1
$$
and
$$
\| \varphi_n \|_{L^{p}} = C(p,s) + o(1) \,.
$$
Thanks to inequality \eqref{eq:dbound2} we can find uniform upper bound on $\|\varphi_n\|_{L^{p_1}}$ and $\|\varphi_n\|_{L^{p_2}}$ for some $p_1<p<p_2$. Therefore, by the well known pqr-Lemma (see \cite{FLL})
$$
\inf_n \left|\{ |\varphi_n|>\eta \}\right| >0 \,.
$$
Now by Lieb's compactness lemma in $\dot H^s$, see Lemma 2.1 in \cite{BFV}, there exists $\varphi \neq 0$ such that  $\varphi_n \rightharpoonup \varphi \in\dot H^s(\R^3)\cap L^p(\R^3)$. Finally, by the non-local Brezis--Lieb lemma for the Coulomb term (see Lemma 2.2 in \cite{BFV}), and  by the Hilbert structure of $\dot H^s(\R^3)$, we prove the existence of a maximizer.
\end{proof}

\end{document}